\documentclass[11pt]{article}
\usepackage{amsmath, amsthm, amssymb, amsfonts, mathrsfs, mathtools}
\usepackage{graphicx}
\usepackage{makeidx}
\usepackage[left=2cm,right=2cm,top=2cm,bottom=2cm]{geometry}
\usepackage{caption}
\usepackage{subcaption}
\usepackage[section]{placeins}
\usepackage{enumitem}
\usepackage{tikz}
\usepackage{stmaryrd}

\usetikzlibrary{decorations.pathreplacing}
\tikzstyle{vertex}=[auto=left,circle,draw=black,fill=white, inner sep=1.5]

\usepackage{hyperref}
\hypersetup{colorlinks=true, citecolor={blue}, linkcolor=blue, filecolor=magenta, urlcolor=cyan}

\newtheorem{theorem}{Theorem}[section]

\newtheorem{lema}[theorem]{Lemma}
\newtheorem{corollary}{Corollary}[theorem]

\title{Spectral properties of Cayley graphs over finite commutative rings}

\author{Priya and Sanjay Kumar Singh\\
Department of Mathematics\\
Indian Institute of Science Education and Research Bhopal, India.\\
priya22@iiserb.ac.in\\ sanjayks@iiserb.ac.in}

\date{}

\begin{document}
	\maketitle
	
	\vspace{-0.3in}
	
\begin{center}{\textbf{Abstract}}\end{center} 
Let $R$ be a finite commutative ring with unity and $x$ be a non-zero element of $R$. In this paper, we calculate the spectrum and energy of the Cayley graph ${\rm Cay}(R,xR^{*})$, and also compute the energy of their compliment graph. Further, we give necessary and sufficient condition for Cayley graph ${\rm Cay}(R,xR^{*})$ to be Ramanujan. 
\noindent

\vspace*{0.3cm}

\section{Introduction}
A \textit{graph} $G$ is an ordered pair $(V(G), E(G))$ consisting of a non-empty vertex set $V(G)$ and an edge set $E(G)$ of unordered pairs of elements of $V(G)$. A graph is finite if $V(G)$ and $E(G)$ both are finite sets. In this paper, we consider only finite graphs. The \textit{$(0,1)$-adjacency matrix} of $G$, denoted by $\mathcal{A}(G)$, is the square matrix $[a_{uv}]$, where $a_{uv}$ is given by 
\[a_{uv} = \left\{ \begin{array}{rl}
	1 &\mbox{ if }
	(u,v)\in E(G) \\ 
	0 &\textnormal{ otherwise.}
\end{array}\right. \] Note that if $G$ has a loop at vertex $v$ then the $v^{th}$ diagonal entry of $\mathcal{A}(G)$ will be $1$. The eigenvalues of $\mathcal{A}(G)$ are called the \textit{eigenvalues} of $G$. If $\lambda_1, \ldots ,\lambda_k$ are eigenvalues of a graph $G$ with multiplicities $m_1,\ldots ,m_k$, respectively, then 
$\text{Spec}(G)=\begin{pmatrix}
\lambda_1& \ldots& \lambda_k\\
m_1 & \ldots & m_k
\end{pmatrix}$ describe the spectrum of $G$. The \textit{energy} of a graph  $G$ is the sum of absolute values of all the eigenvalues of $G$, \textit{i.e.}, if $\lambda_1, \ldots, \lambda_k$ are eigenvalues of graph $G$ with multiplicites $m_1, \ldots, m_k$, respectively, then energy of $G$ is
$$
\mathcal{E}(G) = \sum_{i=1}^k m_i|\lambda_i| .
$$
The energy of graph was first discussed by I. Gutman in \cite{mmmmmNew} and it is widely studied in chemical graph theory. We refer to  \cite{gutman2001energy,liu2022eigenvalues} for a survey on energy of graphs. 

A graph $G$ is said to be $k$-\textit{regular} if $\sum\limits_{ v \in V(G)} a_{uv}=k$ for all $u \in V(G)$, \textit{i.e.}, sum of entries of each row of $\mathcal{A}(G)$ is $k$.  Let $G$ be a $k-$regular graph. We call $G$ to be a Ramanujan graph if $|\lambda(G)| \leq 2\sqrt{k-1}$, for each eigenvalue $\lambda(G)$ of $G$ whose absolute value less than $k$. For detailed survey on Ramanujan graphs we refer to \cite{davidoff2003elementary,murty2020ramanujan}.

 In general, the Cayley graph was defined on a group by taking connection set as any subset of a group. However, in this paper, we consider Cayley graphs only on finite commutative rings with a particular type of connection set.  Let $R$ be a finite commutative ring with unity and $x$ be a non-zero element of $R$. Define $R^*$ to be the set of all units of $R$ and $xR^*:=\{xr \colon r \in R^* \}$. The Cayley graph of $R$ with connection set $xR^*$, denoted by ${\rm Cay}(R,xR^*)$, is a graph with $V({\rm Cay}(R,xR^*))=R$ and $$E({\rm Cay}(R,xR^*))=\{ (u,v)\colon u,v\in R, u-v \in xR^* \}.$$ Note that if $x\in R^*$ then $xR^*=R^*$. The Cayley graph ${\rm Cay}(R,R^*)$ is known as unitary Cayley graph. For a survey on eigenvalues of Cayley graphs, we refer to \cite{liu2022eigenvalues}.

 In 2009, A. Ili{\'c} \cite{ilic2009energy} gave the energy of unitary Cayley graph ${\rm Cay}(\mathbb{Z}_n,\mathbb{Z}_n^*)$  and its compliment. Later on, this result was extended to the unitary Cayley graph ${\rm Cay}(R, R^*)$ by Kiani el al.  \cite{kiani2011energy} in 2011, for a finite commutative ring $R$ with unity. In 2010, A. Droll \cite{droll2010classification} characterized the unitary Cayley graph ${\rm Cay}(\mathbb{Z}_n,\mathbb{Z}_n^{*})$ to be a Ramanujan graph. Again, this characterization was extended to the unitary Cayley graph ${\rm Cay}(R, R^*)$ by X. Liu and S. Zhou \cite{liu2012spectral} in 2012, for a finite commutative ring $R$ with unity. In this paper, we find the spectrum of the Cayley graph ${\rm Cay}(R,xR^*)$. Using the spectrum, we calculate the energy of the Cayley graph ${\rm Cay}(R,xR^*)$ and their compliment graph. Finally, we characterize the commutative ring $R$ for which the Cayley graph ${\rm Cay}(R,xR^*)$ is a Ramanujan graph.

This paper is organized as follows. In Section 2, we present some preliminary notions and results. In section 3, by calculating the spectrum of ${\rm Cay}(R, xR^*)$, we find the energy of ${\rm Cay}(R, xR^*)$. In Section 5, we compute the energy of the compliment graph of  ${\rm Cay}(R, xR^*)$. In the last section, we characterize the commutative ring for which the Cayley graph ${\rm Cay}(R, xR^*)$ is Ramanujan.


\section{Preliminaries}

A finite commutative ring with unity is called \textit{local ring} if it has a unique maximal ideal. Let $R$ be a local ring and $M$ be the maximal ideal. Then, it is known that $R^*=R\setminus M$. It is also well known that every element of a finite commutative ring with unity is either a zero divisor or a unit element of $R$. Therefore, the maximal ideal $M$ is the set of all zero-divisors of $R$. Note that $R^*\cup M$ is a disjoint union of $R$. Now, we have the following known result.

\begin{lema} \cite{akhtar2009unitary} \label{order_of_R}If $R$ is a local ring with $M$ as its maximal ideal, then $|R|, |M|$ and $\frac{|R|}{|M|}$ are all powers of $p$, for some prime $p$.
\end{lema}

By Theorem $8.7$ of \cite{atiyah2018introduction}, every finite commutative ring can be written as a product of finite local rings and this decomposition is unique upto the permutation of local rings. Let $R$ be a  finite commutative ring with unity. Throughout the paper, we will use the following terminologies:
\begin{itemize}
\item We will use the notation $\bold{1}$ to denote the multiplicative identity (unity) of $R$ and $\bold{0}$ to denote the additive identity of $R$.
\item We consider $R =R_{1} \times \cdots \times R_{s}$ such that $$\frac{|R_1|}{m_1} \leq \cdots \leq \frac{|R_s|}{m_s},$$ where $R_i$ is a local ring with maximal ideal $M_i$ of order $m_i$ for each $i=1,\ldots,s$. By Theorem $8.7$ of \cite{atiyah2018introduction}, this decomposition is unique upto the permutation of local rings.
\item We consider the element $x \in R$ as an element of the cartesian product $R_1 \times \cdots \times R_s$, that is $x:=(x_1,\ldots ,x_s)$, where $x_i \in R_i$ for each $i \in \{1,\ldots,s\}$.
\item Using $R^{*}=R_{1}^{*} \times \cdots \times R_{s}^{*}$, we observe that if $x\in R$ then $xR^{*}=x_1R_{1}^{*} \times \cdots \times x_sR_{s}^{*}$. Moreover, $|xR^{*}| = \prod\limits_{i=1}^{s} |x_iR_{i}^{*}|$.
\end{itemize}

Let $R$ be a local ring and $x$ be a nonzero element of $R$. We will use the notation $I_{x}$ to denote an ideal of $R$ generated by $x$ and define $$M_{x}:=I_{x}\setminus xR^{*}.$$ Note that $xR^*\cup M_x$ is a disjoint union of $I_x$. Now we have the following result.
\begin{lema}\label{NewLemmsMx} If $M$ is a maximal ideal of $R$, then $M_x=xM$.
\end{lema}
\begin{proof} By the definition of $M_x$ and using the fact that $R^*\cup M$ is a disjoint union of $R$, we have $M_{x} \subseteq xM$. On the other hand, let $xm$ is an element of $xM$ with $m \in M$. Our claim is that $xm \not\in xR^*$. If $xm \in xR^*$ then  $xm=xu$ for some $u  \in R^{*}$.  We have $x(m-u)=\bold{0}$, and so $m-u$ is a zero devisor. Therefore $u \in M$, which is a contradiction. Hence $xm \not\in xR^*$, equivalent to say $xm \in M_x$.
\end{proof}

Define $A_x:= \{ r \in R~ \colon ~ xr =\bold{0}\}$. The set $A_x$ is known as the \textit{annihilator} of $x$. We observe that $A_x$ is an ideal of $R$. Define $C_p(0)$ as the ring with an additive group that is isomorphic to the cyclic group $\mathbb{Z}_p$ and whose multiplication of any two elements is zero. The following result provides some fundamental properties of $I_x$ and $M_x$ that we are going to utilise in the next sections.

\begin{lema}\label{PropertiesOfMx}
Let $R$ be a local ring with maximal ideal $M$ and $x$ be a nonzero element of $R$. The following statements are true:
\begin{enumerate}[label=(\roman*)]
    \item if $r\in R$ and $xr \in M_x$, then $r \in M$.
    \item $M_x$ is a maximal ideal of $I_x$.
    \item $\frac{|I_x|}{|M_x|} =\frac{|R|}{|M|}$.
    \item if $x$ is unit and  $|M_x|=1$, then $I_x$ is a field. 
    \item if $x$ is non-unit and  $|M_x|=1$, then $I_x = C_p(0)$ for some prime $p$.
    \item if $x$ is non-unit and $x^2 \neq \bold{0}$, then $|I_x| \leq |M_x|^{2}$.
    \item if $|M_x|>1$, then $|I_x| \leq |M_x|^{2}$.
\end{enumerate}
\end{lema}
\begin{proof}
\begin{enumerate}[label=(\roman*)]
    \item If $r \not\in M$ then $r \in R^*$. We get $xr \in xR^*$. By definition of $M_x$, $xr \not\in M_x$. Which is a contradiction.
    \item Clearly, $0\in M_x$. Let $a,b \in M_x$. By Lemma~\ref{NewLemmsMx}, we have $a=xp$ and $b=xq$ for some $p,q \in M$. Using $p+q \in M$ and Lemma~\ref{NewLemmsMx}, $a+b=x(p+q) \in M_x$. Let $xr \in I_x$ and $a\in M_x$ with $r\in R$. By Lemma~\ref{NewLemmsMx}, we have $a=xp$  with some $p \in M$.  Since $M$ is an ideal of $R$, $rxp \in M$. Using Lemma~\ref{NewLemmsMx}, $(xr)a = x(rxp)\in M_x$. Thus $M_x$ is an ideal of $I_x$. Now it remains to show that $M_x$ is a maximal ideal of $I_x$. Let $xu\in xR^{*}$ with $u\in R^*$. It suffices to show that ideal of $I_x$ generated by $M_x \cup \{xu\}$ is $I_x$ itself. For this we will show that every element of $I_x$ belongs to its ideal generated by $M_x \cup \{xu\}$. Let $xr$ be a non zero element of $I_x$ with $r\in R$.  Using the fact that any ideal of $R$ generated by $M \cup \{u\}$ is $R$ itself, we get $r=a_{1}m_{1}+\ldots+a_{t}m_{t}+a_{t+1}u$, where $m_i \in M ~ \text{for all} ~ i\in \{ 1 ,\ldots , t\}$ and $a_i \in R ~ \text{for all} ~ i\in \{ 1 ,\ldots, t+1\}$. We have $$xr=x a_{1}m_{1}+\ldots + xa_{t}m_{t}+xa_{t+1}u = a_{1}xm_{1}+\ldots+a_{t}xm_{t}+a_{t+1}xu.$$ Therefore $xr$ belongs to the ideal of $I_x$ generated by $M_x \cup \{xu\}$. Hence $M_x$ is a maximal ideal of $I_x$.

    \item If $x$ is an unit element of $R$, then $I_x =R $ and  $M_x=M$. And so the result holds. Assume that $x $ is a non unit element of $R$. Define a map $$\phi : \frac{R}{M} \to \frac{|I_x|}{|M_x|}~~\textnormal{ such that }~~\phi(a+M)= ax+M_x.$$ If $a+M=b+M$ then $a-b \in M$. By Lemma~\ref{NewLemmsMx}, we get $(a-b)x \in M_x$. It implies that $ax+M_x=bx+M_x$. Thus $\phi$ is well-defined. If $\phi(a+M)=\phi(b+M)$ then $ax+M_x=bx+M_x$. This implies $ax-bx\in M_x$, and so $(a-b)x \in M_x$. By Part $(i)$, we get  $a-b \in M$, which means $a+M=b+M$. Thus $\phi$ is injective. Let $ax+M_x$ be a nonzero element of quotient ring $ \frac{I_x}{M_x}$. Then $ax \notin M_x$, and so $ax \in xR^{*}$. By Part $(i)$, we get $a \notin M$. Which means $a+M$ is a nonzero element of the ring $\frac{R}{M}$ and $\phi(a+M)=ax+M_x$. Thus $\phi$ is bijective. Hence $\frac{|I_x|}{|M_x|} =\frac{|R|}{|M|}$.

    \item If $x$ is unit and  $|M_x|=1$, then $I_x=R$ and $I_x$ has unique maximal trivial ideal. Now the proof follows.

    \item Assume that $x$ is non-unit and  $|M_x|=1$. If $x^2 \neq \bold{0}$ then $x^2 \in M_x$. Therefore $|M_x| >1$, which is not possible. Thus $x^2 = \bold{0}$. Therefore, $I_x$ has no proper, non-trivial ideal and product of any two elements of $I_x$ is zero. By Lemma~\ref{order_of_R}, we have $|I_x|=p^n$ for some prime $p$ and $n\in \mathbb{N}$. By Cauchy's theorem of Abelian group, we get $n=1$ because $I_x$ has no proper, non-trivial ideal. Thus $|I_x|=p$, and so $I_x = C_p(0)$.

    \item Assume that $x$ is non-unit and $x^2 \neq \bold{0}$. Then $x\in M$ and $xr\in M$ for all $r\in R$. We observe that $A_x \cap M_x$ is an ideal of $I_x$. Define $$ \phi : \frac{I_x}{A_x \cap M_x} \to M_x ~~ \textnormal{ such that } ~~\phi(xr+A_x\cap M_x)= x^2r,$$ where $r\in R$. Since $xr\in M$, $x^2r\in M_x$ by Lemma~\ref{NewLemmsMx}. If $xr+A_x\cap M_x = xr'+A_x \cap M_x$ then $xr-xr' \in A_x \cap M_x$. Therefore $x^2(r-r') =0$, which means $x^2r=x^2r'$. Therefore $\phi(xr+A_x\cap M_x) =\phi (xr'+A_x \cap M_x)$. Thus $\phi$ is a well-defined. If $\phi(xr+A_x\cap M_x) =\phi (xr'+A_x \cap M_x)$ then $x^2r = x^2r'$. Therefore $x^2(r-r')=0$, it means $r-r'$ is a zero-divisor of $R$. We have $r-r' \in M$, and so $x(r-r') \in M_x$. Using $x^2(r-r')=0$, we obtain $x(r-r') \in A_x$. So $x(r-r') \in A_x \cap M_x$, it implies $xr+A_x \cap M_x = xr'+A_x \cap M_x$. Thus $\phi$ is injective map. It implies that $|I_x| \leq |M_x||A_x \cap M_x|$. Now the proof follows from $|A_x \cap M_x| \leq |M_x|$.

    \item Assume that $|M_x|>1$. If $x$ is unit element of $R$ then $I_x=R$ and $M_x=M$. Now the proof follows from \cite{ganesan1964properties}. Assume that $x$ is non-unit element of $R$. We will split the proof into two cases.\\
    \textbf{Case 1}: If $x^{2}\neq 0$ then proof follows from Part $(vi)$.
    
    \textbf{Case 2}: If $x^{2} = 0$ then  the product of any two elements of $I_x$ is zero. By Lemma~\ref{order_of_R}, we have $|I_x|=p^n$ for some prime $p$ and $n\in \mathbb{N}$. Using the fact that every maximal subgroup of an abelian group of order $p^n$ has order $p^{n-1}$ and $M_x$ is maximal subgroup of $I_x$, we get $|M_x|=p^{n-1}$. By $|M_x|>1$, we have $n \geq 2$. Thus the proof follows from the inequality $p^n \leq p^{2(n-1)}$ whenever $n \geq 2$.
\end{enumerate}
\end{proof}


\section{Energy of the Cayley graph ${\rm Cay}(R, xR^{*})$} 

In this section, we first calculate the spectrum of the Cayley graph ${\rm Cay}(R,xR^{*})$. Using that, we find the energy of ${\rm Cay}(R,xR^{*})$.

Let $G =(V(G),E(G))$ and $H=(V(H),E(H))$ be two graphs. Then tensor product $G\otimes H $ is the graph with vertex set $V(G)\times V(H)$, and $((u,v),(u',v')) \in E(G\otimes H)$ if and only if $(u,u') \in E(G)$ and $(v,v') \in E(H)$. In the next result, we express a Cayley graph in the tensor products of a Cayley graphs over local rings.

\begin{theorem}\label{R_as_a_product of local rings} Let $R$ be a finite commutative ring and $x$ be a nonzero element of $R$. Then the following statements are true:
\begin{enumerate}[label=(\roman*)]
    \item The Cayley graph ${\rm Cay}(R,xR^{*})$ is $| xR^{*} |$-regular.
    \item Both Cayley graphs ${\rm Cay}(R,xR^{*})$ and ${\rm Cay}(R_{1},x_{1}R_{1}^{*})\otimes \cdots \otimes {\rm Cay}(R_{s},x_{s}R_{s}^{*}) $ are isomorphic.
\end{enumerate}
\end{theorem}
\begin{proof}
\begin{enumerate}[label=(\roman*)]
    \item The proof follows from $\sum\limits_{ v \in R } a_{uv}=\sum\limits_{ s \in xR^{*} } a_{u \hspace{0.05cm} u-s}=| xR^{*} |$ for all $u \in V(G)$.
    \item Let  $G={\rm Cay}(R,xR^{*})$ and $G_i = {\rm Cay}(R_{i},x_{i}R_{i}^{*})$ for each $i=1,\ldots,s$. By our assumptions, we have $ R= R_1\times \cdots \times R_s$ and $xR^{*}=x_1R_1^{*}\times \ldots \times x_sR_s^{*}$. Therefore, the vertex set of both graphs $G$ and $G_1 \otimes \cdots \otimes G_s$ are equal. Let $a,b \in R$. Note that $a-b\in xR^{*}$ if and only if $a_i-b_i\in x_iR_i^{*}$ for each $i=1,\ldots , s$. Hence, both graphs $G$ and $G_1 \otimes \cdots \otimes G_s$ are isomorphic.
\end{enumerate} 
\end{proof}

In the last theorem, we assume that  $x$ is a nonzero element of $R$, so the Cayley graph ${\rm Cay}(R,xR^{*})$ has no loops. But, there may be some $1\leq j \leq s$ such that $x_j$ is a zero element of the local ring $R_j$. In that case, the Cayley graph ${\rm Cay}(R_{j},x_{j}R_{j}^{*})$ will have a loop on each vertex and there will be no edges between any two distinct vertices, and so the $(0,1)$-adjacency matrix of ${\rm Cay}(R_{j},x_{j}R_{j}^{*})$ will be identity matrix.  A \textit{component} of a graph $G$ is a connected subgraph of $G$  such that it is not subgraph of any larger connected subgraph of $G$. The next result provides some combinatorial properties of components of ${\rm Cay}(R, xR^{*})$.

\begin{lema}\label{Isomorphic_Component_Partite}
Let $R$ be a local ring and $x$ be a non zero element of $R$. Then the following statements are true:
\begin{enumerate}[label=(\roman*)]
	\item ${\rm Cay}(I_x,xR^{*})$ is a complete $\frac{|I_x|}{|M_x|}$-partite graph whose partite sets are cosets of $\frac{I_x}{M_x}$.
	\item ${\rm Cay}(I_x,xR^{*})$ is isomorphic to $G_{[z]}$ for each $z\in R$, where $V(G_{[z]})=z+I_x$ and $E(G_{[z]})=\{(z+a,z+b)~ \colon ~ a-b\in xR^{*}\}$.
	\item If $a,b \in R$, there is a path between $a$ and $b$ in ${\rm Cay}(R, xR^{*})$ if and only if $a$ and $b$ are lies in same coset of $\frac{R}{I_x}$.
	\item ${\rm Cay}(R, xR^{*})$ has exactly $\frac{|R|}{|I_x|}$ components. 
	\item ${\rm Cay}(I_x, xR^{*})$ is isomorphic to each component of ${\rm Cay}(R, xR^{*})$.
\end{enumerate} 
\end{lema}
\begin{proof}
\begin{enumerate}[label=(\roman*)]
        \item Observe that $a$ and $b$ are adjacent in ${\rm Cay}(I_x,xR^{*})$ if and only if $a-b\in xR^{*}$, equivalent to say, $a-b\notin M_x$. Thus $a$ and $b$ are adjacent if and only if they are in different cosets of $\frac{I_x}{M_x}$. The vertex set of ${\rm Cay}(I_x,xR^{*})$ can be partitioned into $\frac{|I_x|}{|M_x|}$ different independent sets and there is an edge between every pair of vertices from different independent sets. This implies ${\rm Cay}(I_x,xR^{*})$ is a complete $\frac{|I_x|}{|M_x|}$-partite graph. 
        \item Define a bijective mapping $\theta \colon V({\rm Cay}(I_{x},xR^{*}))  \rightarrow V(G_{[z]})$ such that $\theta(a)=z+a$ for all $a \in I_x$. It is clear from the definition of $G_{[z]}$ that $a$ and $b$ are adjacent in ${\rm Cay}(I_x, xR^{*})$ if and only if $z+a$ and $z+b$ are adjacent in $G_{[z]}$. Thus $\theta$ is an isomorphism between $G_{[z]}$ and ${\rm Cay}(I_x,xR^{*})$.

	\item Let $a,b \in R$. Assume that there is a path $P$ between $a$ and $b$ in ${\rm Cay}(R, xR^{*})$. Then there exist a sequence of vertex $a_1,\ldots, a_{k-1}$ such that $P= a,a_1,\ldots,a_{k-1},b$. We get $b=a+s_1+\ldots + s_{k}$, for some $s_1, \ldots , s_{k} \in xR^*$. Using $xR^* \subseteq I_x$, we get $b\in a+I_x$. Thus, $a$ and $b$ are lies in same coset of $\frac{R}{I_x}$. Conversly, assume that $a$ and $b$ are lies in same coset of $\frac{R}{I_x}$. We have $a,b \in z+I_x$ for some $z\in R$. Thus $a$ and $b$ are two vertex of the graph $G_{[z]}$, as defined in Part $(ii)$. Using Part $(i)$ and Part $(ii)$, $G_{[z]}$ is a complete $\frac{|I_x|}{|M_x|}$-partite graph, and so $a$ and $b$ are connected by a path in $G_{[z]}$. Hence, there is a path between $a$ and $b$ in ${\rm Cay}(R, xR^{*})$.
	\item Using Part $(iii)$, a subgraph induced by the vertex set $z+I_x$ is a component. Hence, the number of components are equal to the size of $\frac{R}{I_x}$.
	\item Using Part $(iii)$, a subgraph induced by the vertex set $z+I_x$ is a component, and that component is isomorphic to $G_{[z]}$. Now, the proof follows from Part $(ii)$.
\end{enumerate} 
\end{proof}

Let $G$ be a loopless graph. The \textit{compliment} of $G$, denoted by $\overline{G}$, is a loopless graph with the same vertex set as of $G$ and  two distinct vertices are adjacent in $\overline{G}$ if they are not adjacent in $G$. 

\begin{theorem} \cite{west2001introduction,godsil2001algebraic}   \label{specofG_and_Gbar}
		If a graph $G$ be a $k-$regular graph with $n$ vertices, then $G$ and $\overline{G}$ have the same eigenvectors. $k$ and $n-k-1$ are the eigenvalues of $G$ and  $\overline{G}$, respectively, associated to eigenvector $J_{n\times 1}$ whose all entries are $1$. And if $\lambda$ is eigenvalue of $G$ with eigenvector $ \overline{x}\neq J_{n\times 1}$ then $-1-\lambda$ is associated eigenvalue of $\overline{G}$.
\end{theorem}

\begin{lema} \label{SpecofCayIx} Let $R$ be a local ring and $x$ be a non zero element of $R$. Then the spectrum of ${\rm Cay}(I_x,xR^{*})$ is $\begin{pmatrix}
|xR^{*}|&-|M_x|&0\\
1&\frac{|I_x|}{|M_x|}-1&\frac{|I_x|}{|M_x|}(|M_x|-1)
\end{pmatrix}$. 
\end{lema}

\begin{proof}
Let $G={\rm Cay}(I_x,xR^{*})$. By Lemma~\ref{Isomorphic_Component_Partite}, $G$ is a complete $\frac{|I_x|}{|M_x|}$- partite graph. Therefore $\overline{G}$ has $\frac{|I_x|}{|M_x|}$ components and each component is a complete graph with $M_x$ vertices. Let $J$ be the $|M_x|\times |M_x|$ matrix with each entry is $1$. Note that the $(0,1)$-adjacency matrix of $\overline{G}$ is a block diagonal matrix with $\frac{|I_x|}{|M_x|}$ blocks and each block matrix is equal to $J-I$, where $I$ is $|M_x|\times |M_x|$ identity matrix. Since the spectrum of $J-I$ is $\begin{pmatrix}
|M_x|-1&-1\\
1&|M_x|-1
\end{pmatrix}$, the spectrum of $\overline{G}$ is $\begin{pmatrix}
|M_x|-1&-1\\
\frac{|I_x|}{|M_x|}& \frac{|I_x|}{|M_x|}(|M_x|-1)
\end{pmatrix}$. Using Theorem \ref{specofG_and_Gbar}, we observe the following things: 
\begin{enumerate}[label=(\roman*)]
\item $G$ has eigenvalue $|I_x|-(|M_x|-1)-1=|I_x|-|M_x|=|xR^{*}|$ with multiplicity $1$ corresponding to the eigenvector $J_{n\times 1}$.
\item $-(|M_x|-1)-1=-|M_x|$ is an eigenvalue with multiplicity $\frac{|I_x|}{|M_x|}-1$.
\item $-(-1)-1=0$ is an eigenvalue with multiplicity $\frac{|I_x|}{|M_x|}(|M_x|-1)$.
\end{enumerate} 
Therefore, the spectrum of ${\rm Cay}(I_x,xR^{*})$ is $\begin{pmatrix}
|xR^{*}|&-|M_x|&0\\
1& \frac{|I_x|}{|M_x|}-1 & \frac{|I_x|}{|M_x|}(|M_x|-1)
\end{pmatrix}$.
\end{proof}.

Using Part $(iv)$ and Part $(v)$ of Lemma~\ref{Isomorphic_Component_Partite},  ${\rm Cay}(R, xR^{*})$ has exactly $\frac{|R|}{|I_x|}$ components and each component is isomorphic to  ${\rm Cay}(I_x,xR^{*})$. Therefore the $(0,1)$-adjacency matrix of ${\rm Cay}(R, xR^{*})$ is a block diagonal matrix with $\frac{|R|}{|I_x|}$ blocks and each block matrix is equal to the $(0,1)$-adjacency matrix of ${\rm Cay}(I_x,xR^{*})$.  Now, we have the following result.

\begin{lema}\label{eigenvalue_corr_to_local_ring}
Let $R$ be a local ring and $x$ be a non-zero element of $R$. Then the spectrum of ${\rm Cay}(R,xR^{*})$ is $\begin{pmatrix}
|xR^{*}| & -|M_x| & 0\\
\frac{|R|}{|I_x|} & \frac{|R| |xR^*|}{|I_x| |M_x|} &\frac{|R|}{|M_x|}(|M_x|-1)
\end{pmatrix}$.
\end{lema}
\begin{proof} The proof follows from Part $(iv)$ and Part $(v)$ of Lemma~\ref{Isomorphic_Component_Partite} and Lemma~\ref{SpecofCayIx}.
\end{proof}

Recall that the energy of a graph $G$ is the sum of absolute values of all the eigenvalues of $G$. The next result is a consequence of Theorem~\ref{eigenvalue_corr_to_local_ring} that provides a formula to calculate the energy of ${\rm Cay}(R,xR^{*})$.

\begin{corollary}\label{energy_of_local_ring}
Let $R$ be a local ring and $x$ be a non-zero element of $R$. Then the energy of ${\rm Cay}(R,xR^{*})$ is $2|xR^{*}|\frac{|R|}{|I_x|}$.
\end{corollary}

\begin{proof}
The proof follows from $|xR^{*}|\frac{|R|}{|I_x|}+|M_x|\frac{|xR^{*}||R|}{|M_x||I_x|}=2|xR^{*}|\frac{|R|}{|I_x|}$.
\end{proof}

\begin{theorem}\cite{west2001introduction}\label{Eigenvalues_of_tensor_product_of_graphs}
Let $G$ and $H$ be two graphs. If $\lambda_1, \ldots, \lambda_n$ are the eigenvalues of $G$ and $\mu_1,\ldots,\mu_m $ are the eigenvalues of $H$ then $\lambda_{i}\mu_{j}$,where $1\leq i\leq n$ and $1\leq j \leq m$ are the eigenvalues of $G\otimes H$. 
\end{theorem}

Let $R$ be a finite commutative ring, $x$ be a non-zero element of $R$, and $P=\{ i \colon  i \in \{ 1, \ldots, s \} \text{ and } x_i \neq \bold{0}\}$. By Theorem \ref{R_as_a_product of local rings}, both graphs ${\rm Cay}(R,xR^{*})$ and $\otimes_{i=1}^{s} {\rm Cay}(R_{i},x_{i}R_{i}^{*})$ are isomorphic, and so they have same set of eigenvalues with same multiplicity. Note that if $i \not \in P$ then the $(0,1)$-adjacency matrix of ${\rm Cay}(R_{i},x_{i}R_{i}^{*})$ is square identity matrix of size $|R_i|$. Therefore, the $(0,1)$-adjacency matrix of $\otimes_{i\in P^c} {\rm Cay}(R_{i},x_{i}R_{i}^{*})$ is square identity matrix of size $\prod\limits_{i\in P^{c}}|R_i|$. By Theorem \ref{Eigenvalues_of_tensor_product_of_graphs}, both graphs ${\rm Cay}(R,xR^{*})$ and $\otimes_{i\in P} {\rm Cay}(R_{i},x_{i}R_{i}^{*})$ have same set of eigenvalues, but the multiplicity of each eigenvalue in ${\rm Cay}(R,xR^{*})$ is equal to  $\prod\limits_{i\in P^{c}}|R_i|$ times the multiplicity of corresponding eigenvalue in $\otimes_{i\in P} {\rm Cay}(R_{i},x_{i}R_{i}^{*})$. Observe that $|I_{x_i}| = |M_{x_i}|+|x_iR_i^{*}|$ for each $x_i\neq \bold{0}$. Therefore,
\begin{equation} \label{card_of_xR*}
|xR^{*}| = \prod_{\substack{i=1 \\  x_i \neq \bold{0}}}^s |x_iR_{i}^{*}| = \prod_{\substack{i=1 \\  x_i \neq \bold{0}}}^{s}(|I_{x_i}|-|M_{x_i}|)=\prod_{\substack{i=1 \\  x_i \neq \bold{0}}}^{s}|M_{x_i}|\prod_{\substack{i=1 \\  x_i \neq \bold{0}}}^{s}\bigg(\frac{|I_{x_i}|}{|M_{x_i}|}-1\bigg).
\end{equation}

\noindent The result appears as follows.

\begin{theorem}\label{eigenvalue_of_cay(R,xR)} Let $R$ be a finite commutative ring, $x$ be a non-zero element of $R$, and $P=\{ i \colon  i \in \{ 1, \ldots, s \} \text{ and } x_i \neq \bold{0}\}$. Then eigenvalues of ${\rm Cay}(R,xR^{*})$ are
\begin{enumerate}[label =(\roman*)]
\item $(-1)^{|C|}\frac{|xR^{*}|}{\prod\limits_{i\in C}\frac{|x_iR_i^{*}|}{|M_{x_i}|}}$ with multiplicity $\frac{|R|}{|I_x|} \prod\limits_{i\in C}\frac{|x_iR_{i}^{*}|}{|M_{x_i}|}$ for all subset $C$ of $P$, where $P^{c}=\{1,2,\cdots,s\}\setminus P$.
\item 0 with multiplicity $|R|- \frac{|R|}{|I_x|} \prod\limits_{i\in P}\left(1+\frac{|x_iR_i^{*}|}{|M_{x_i}|}\right)$.
\end{enumerate}
\end{theorem}

\begin{proof} 
\begin{enumerate}[label=(\roman*)]
\item Let $C$ be a non-empty subset of $P$. By Theorem \ref{Eigenvalues_of_tensor_product_of_graphs}, $(-1)^{|C|}\prod\limits_{i\in C}|M_{x_i}| \prod\limits_{i\in P\setminus C}|x_iR_i^{*}|$ is a non-zero eigenvalue of ${\rm Cay}(R,xR^{*})$ with multiplicity $\prod\limits_{i\in C}\frac{|R_i| |x_iR_i^{*}|}{|I_{x_i}| |M_{x_i}|}\prod\limits_{i\in P\setminus C}\frac{|R_i|}{|I_{x_i}|}\prod\limits_{i\in P^{c}}|R_i|$. We have $$(-1)^{|C|}\prod\limits_{i\in C}|M_{x_i}| \prod\limits_{i\in P\setminus C}|x_iR_i^{*}|  =   (-1)^{|C|}\frac{|xR^{*}|}{\prod\limits_{i\in C}\frac{|x_iR_i^{*}|}{|M_{x_i}|}}$$ and 
\begin{align}
\prod\limits_{i\in C}\frac{|R_i| |x_iR_i^{*}|}{|I_{x_i}| |M_{x_i}|}\prod\limits_{i\in P\setminus C}\frac{|R_i|}{|I_{x_i}|}\prod\limits_{i\in P^{c}}|R_i|   &=          \prod\limits_{i\in P^{c}}|R_i|\prod\limits_{i\in P}\frac{|R_i|}{|I_{x_i}|}\prod\limits_{i\in C}\frac{|x_iR_{i}^{*}|}{|M_{x_i}|} \nonumber \\
&= \frac{|R|}{|I_x|} \prod\limits_{i\in C}\frac{|x_iR_{i}^{*}|}{|M_{x_i}|}. \nonumber
\end{align}
In particular, if $C=\phi$ then $\prod\limits_{i\in P} |x_{i}R_{i}^{*}| $ is an eigenvalue with multiplicity $\prod\limits_{i\in P^{c}}|R_i|\prod\limits_{i\in P}\frac{|R_i|}{|I_{x_i}|}$. Again, we have
\begin{align}
\prod\limits_{i\in P} |x_{i}R_{i}^{*}|  = |xR^*| \textnormal{ and } \prod\limits_{i\in P^{c}}|R_i|\prod\limits_{i\in P}\frac{|R_i|}{|I_{x_i}|} = \frac{|R|}{|I_x|}. \nonumber
\end{align}

\item  By Part $(i)$, the number of non-zero eigenvalues of ${\rm Cay}(R,xR^{*})$ is $$\frac{|R|}{|I_x|}       +\sum\limits_{C\subset P,C\neq \phi}\left( \frac{|R|}{|I_x|} \prod\limits_{i\in C}\frac{|x_iR_{i}^{*}|}{|M_{x_i}|} \right).$$ 

We have  $$ \frac{|R|}{|I_x|} \left(1+\sum\limits_{C\subset P, C\neq \phi}\prod\limits_{i\in C } \frac{|x_{i}R_{i}^{*}|}{|M_{x_i}|}\right)= \frac{|R|}{|I_x|} \prod\limits_{i\in P}\left(1+\frac{|x_{i}R_{i}^{*}|}{|M_{x_i}|}\right).$$ Therefore, $0$ is an eigenvalue with multiplicity $|R|-\frac{|R|}{|I_x|} \prod\limits_{i\in P}\left(1+\frac{|x_{i}R_{i}^{*}|}{|M_{x_i}|}\right)$.
\end{enumerate}
\end{proof}

In Theorem \ref{Eigenvalues_of_tensor_product_of_graphs}, we have $\mathcal{E}(G\otimes H) = \mathcal{E}(G)\mathcal{E}(H)$. Now, we have the last result of this section.

\begin{theorem} \label{energy_for_arb_R}
Let $R$ be a finite commutative ring, $x$ be a non-zero element of $R$, and $P=\{ i \colon  i \in \{ 1, \ldots, s \} \text{ and } x_i \neq \bold{0}\}$. Then the energy of ${\rm Cay}(R,xR^{*})$ is $2^{|P|}\frac{|R||xR^{*}|}{|I_x|}$.
\end{theorem}
\begin{proof} We have $$\mathcal{E} ( {\rm Cay}(R,xR^{*})) = \prod_{i=1}^{s} \mathcal{E} ( {\rm Cay}(R_{i},x_{i}R_{i}^{*})).$$ Since the $(0,1)$-adjacency matrix of $\otimes_{i\in P^c} {\rm Cay}(R_{i},x_{i}R_{i}^{*})$ is square identity matrix of size $\prod\limits_{i\in P^{c}}|R_i|$, we have $\prod\limits_{i\in P^c} \mathcal{E} ({\rm Cay}(R_{i},x_{i}R_{i}^{*})) = \prod\limits_{i\in P^{c}}|R_i|$. Therefore, 
\begin{align}
\mathcal{E} ({\rm Cay}(R,xR^{*})) &= \prod_{i\in P} \mathcal{E} ( {\rm Cay}(R_{i},x_{i}R_{i}^{*})) \times \prod_{i\in P^c} \mathcal{E} ( {\rm Cay}(R_{i},x_{i}R_{i}^{*})) \nonumber \\
&= \prod_{i\in P}  2|x_iR_i^{*}|\frac{|R_i|}{|I_{x_i}|} \times \prod\limits_{i\in P^{c}}|R_i| \nonumber \\
&= 2^{|P|}\frac{|R||xR^{*}|}{|I_x|}. \nonumber
\end{align}
In the third equality, we use $|R|=\prod\limits_{i=1}^s |R_i|$, $|xR^*|=\prod\limits_{i\in P} |x_iR_i^*|$, and $|I_x|= \prod\limits_{i\in P}|I_{x_i}|$.
\end{proof}



\section{Energy of the complement graph of ${\rm Cay}   (R,xR^{*})$}

In this section we compute the energy of the compliment  graph of ${\rm Cay}   (R,xR^{*})$.

\begin{theorem}
Let $R$ be a finite commutative ring and $x$ be a non-zero element of $R$. Then energy of $\overline{{\rm Cay}(R,xR^{*})}$ is $$2(|R|-|xR^{*}|-1)+   \frac{|R|}{|I_x|} \left[ 2^{|P|}|xR^*| +  \prod\limits_{i\in P} \left( 1-\frac{|x_iR_{i}^{*}|}{|M_{x_i}|}\right) - \prod_{i\in P}\left(1+\frac{|x_iR_i^{*}|}{|M_{x_i}|}\right)  \right] ,$$ where $P=\{ i \colon  i \in \{ 1, \ldots, s \} \text{ and } x_i \neq \bold{0} \}$.
\end{theorem}
\begin{proof}
Let $\lambda_1, \lambda_2,\ldots,\lambda_{|R|}$ be the eigenvlaues of ${\rm Cay}(R,xR^{*})$ and $\lambda_1=|xR^{*}|$. Theorem~\ref{specofG_and_Gbar} implies that $|R|-|xR^{*}|-1,-\lambda_2-1,\ldots,-\lambda_{|R|}-1$ are eigenvalues of ${\rm Cay}(R,xR^{*})$. We have
\begin{align}
\mathcal{E}(\overline{{\rm Cay}(R,xR^{*}})&=|R|-|xR^{*}|-1+\sum_{i\neq 1}|-1-\lambda_i| \nonumber \\
&=|R|-|xR^{*}|-1+\sum_{i\neq 1}|\lambda_i+1| \nonumber \\
&=|R|-|xR^{*}|-1+\sum_{i\neq 1,\lambda_i\neq 0}|\lambda_i+1|+\sum_{i\neq 1,\lambda_i= 0}|\lambda_i+1|. \label{Eq11-Energy_Compliment}
\end{align}
By Theorem \ref{eigenvalue_of_cay(R,xR)}, we have 
\begin{align}
\sum_{i\neq 1,\lambda_i= 0}|\lambda_i+1| =|R|- \frac{|R|}{|I_x|} \prod_{i\in P}\left(1+\frac{|x_iR_i^{*}|}{|M_{x_i}|}\right). \label{Eq1-Energy_Compliment}
\end{align}

Let $C$ be a non-empty subset of $P$. Using Theorem~\ref{eigenvalue_of_cay(R,xR)}, $(-1)^{|C|}\frac{|xR^{*}|}{\prod\limits_{i\in C}\frac{|x_iR_i^{*}|}{|M_{x_i}|}}$ is a non-zero eigenvalue with multiplicity $\frac{|R|}{|I_x|} \prod\limits_{i\in C}\frac{|x_iR_{i}^{*}|}{|M_{x_i}|}$. Similarly, if $C=\phi$ then $|xR^{*}|$ is also a non-zero eigenvalue of ${\rm Cay}(R,xR^{*})$ with multiplicity $\frac{|R|}{|I_x|}$. We have
\begin{align}
\sum_{i\neq 1,\lambda_i\neq 0}|\lambda_i+1| &= \sum_{C\subset P,C\neq \phi} \Bigg |(-1)^{|C|}\frac{|xR^{*}|}{\prod\limits_{i\in C}\frac{|x_iR_i^{*}|}{|M_{x_i}|}}+1\Bigg |  \left( \frac{|R|}{|I_x|}  \prod\limits_{i\in C}\frac{|x_iR_{i}^{*}|}{|M_{x_i}|}\right)  +\left(|xR^{*}|+1\right)\left(\frac{|R|}{|I_x|} -1\right)\nonumber \\
&= \sum_{C\subset P,C\neq \phi} \Bigg |(-1)^{|C|}  \frac{|R| |xR^*|}{|I_x|} + \frac{|R|}{|I_x|}  \prod\limits_{i\in C}\frac{|x_iR_{i}^{*}|}{|M_{x_i}|}\Bigg |  +\left(|xR^{*}|+1\right)\left(\frac{|R|}{|I_x|} -1\right)\nonumber \\
&=  (2^{|P|}-1)   \frac{|R| |xR^*|}{|I_x|} + \sum_{C\subset P,C\neq \phi} (-1)^{|C|} \frac{|R|}{|I_x|}  \prod\limits_{i\in C}\frac{|x_iR_{i}^{*}|}{|M_{x_i}|}  +\left(|xR^{*}|+1\right)\left(\frac{|R|}{|I_x|} -1\right)\nonumber \\
&= \frac{|R|}{|I_x|} \left[ 2^{|P|}|xR^*| +  \sum_{C\subset P,C\neq \phi} (-1)^{|C|} \prod\limits_{i\in C}\frac{|x_iR_{i}^{*}|}{|M_{x_i}|} +1\right] -\left(|xR^{*}|+1\right) \nonumber \\
&= \frac{|R|}{|I_x|} \left[ 2^{|P|}|xR^*| +  \prod\limits_{i\in P} \left( 1-\frac{|x_iR_{i}^{*}|}{|M_{x_i}|}\right) \right] -\left(|xR^{*}|+1\right). \label{Eq2-Energy_Compliment}
\end{align}

\noindent Apply Equation~(\ref{Eq1-Energy_Compliment}) and Equation~(\ref{Eq2-Energy_Compliment}) in Equation~(\ref{Eq11-Energy_Compliment}), we get
\begin{align}
\mathcal{E}(\overline{{\rm Cay}(R,xR^{*}})&=(|R|-|xR^{*}|-1)+   \frac{|R|}{|I_x|} \left[ 2^{|P|}|xR^*| +  \prod\limits_{i\in P} \left( 1-\frac{|x_iR_{i}^{*}|}{|M_{x_i}|}\right) \right] -\left(|xR^{*}|+1\right)      \nonumber\\
&~~~~~+ |R|- \frac{|R|}{|I_x|} \prod_{i\in P}\left(1+\frac{|x_iR_i^{*}|}{|M_{x_i}|}\right) \nonumber \\
&=2(|R|-|xR^{*}|-1)+   \frac{|R|}{|I_x|} \left[ 2^{|P|}|xR^*| +  \prod\limits_{i\in P} \left( 1-\frac{|x_iR_{i}^{*}|}{|M_{x_i}|}\right) - \prod_{i\in P}\left(1+\frac{|x_iR_i^{*}|}{|M_{x_i}|}\right)  \right].   \nonumber 
\end{align}
\end{proof}


\section{Ramanujan Graphs}

In this section, we give a characterization for the commutative ring $R$ for which the Cayley graph ${\rm Cay}(R,xR^{*})$ is Ramanujan. First, we will see known classification of finite rings of order $p$ and $p^2$.

\begin{lema} \label{rings_of_order_P} \cite{fine1993classification} $\mathbb{Z}_p$ and $C_p(0)$ are the only rings of prime order $p$, upto isomorphism.
\end{lema}
 
\begin{theorem} \cite{fine1993classification} \label{rings_of_order_p2}
There are exactly 11 rings, upto isomorphism, of order $p^2$, where $p$ is a prime. These rings are given by the following presentations:
\begin{align*}
\mathcal{A}_{p^2} &= \langle a; ~ p^2a=0, ~a^2=a ~\rangle ~= ~ \mathbb{Z}_{p^2}\\
\mathcal{B}_{p^2} &=\langle a;~ p^2a =0, ~ a^2 =pa \rangle\\
\mathcal{C}_{p^2}& = \langle a;~ p^2=0,~ a^2=0 \rangle =C_{p^2}(0)\\
\mathcal{D}_{p^2}&=\langle a,b;~ pa=pb=0,~a^2=a,~b^2=b,~ab=ba=0 \rangle =\mathbb{Z}_p + \mathbb{Z}_p\\
\mathcal{E}_{p^2} & =\langle a,b;~ pa=pb=0,~a^2=a,~b^2=b,~ab=a,~ba=b \rangle\\
\mathcal{F}_{p^2}&=\langle a,b;~ pa=pb=0,~a^2=a,~b^2=b,~ab=b,~ba=a \rangle\\
\mathcal{G}_{p^2}&=\langle a,b;~ pa=pb=0,~a^2=0,~b^2=b,~ab=a,~ba=a \rangle\\
\mathcal{H}_{p^2}&=\langle a,b;~ pa=pb=0,~a^2=0,~b^2=b,~ab=ba=0 \rangle =\mathbb{Z}_p+C_p(0)\\
\mathcal{I}_{p^2}&=\langle a,b;~ pa=pb=0,~a^2=b,~ab=0 \rangle\\
\mathcal{J}_{p^2}&=\langle a,b;~ pa=pb=0,~a^2= b^2=0 \rangle =C_p \times C_p(0)\\
\mathcal{K}_{p^2}&= GF(p^2) =~\text{finite field of order}~ p^2
\end{align*}
\end{theorem}

In the next result, we characterize the commutative ring $R$ for which the Cayley graph ${\rm Cay}(R,xR^{*})$ is Ramanujan by assuming $x_i$ to be non-zero for all $i=1,\ldots ,s$.

\begin{theorem} \label{Ramanujan_characterization}
Let $R$ be a finite commutative ring and $x$ be a non-zero element of $R$. If $x_i$ is non-zero for all $i=1,\ldots ,s$, then ${\rm Cay}(R,xR^{*})$ is Ramanujan if and only if one of the following holds
\begin{enumerate}[label =(\roman*)]
\item $\frac{I_{x_i}}{M_{x_i}}  = \left\{ \begin{array}{rl}
	\mathbb{F}_2 &\mbox{ if }
	x_i \in R_i^* \\ 
	C_2(0) &\textnormal{ otherwise}
\end{array}\right.$ for $i=1,\ldots, s$.
\item $I_{x_i} = \left\{ \begin{array}{rl}
	\mathbb{F}_2 &\mbox{ if }
	x_i \in R_i^* \\ 
	C_2(0) &\textnormal{ otherwise}
\end{array}\right.$ for $i = 1, \ldots, s-3$ and 
$I_{x_i} = \left\{ \begin{array}{rl}
	\mathbb{F}_3 &\mbox{ if }
	x_i \in R_i^* \\ 
	C_3(0) &\textnormal{ otherwise}
\end{array}\right.$ for $i=s-2,s-1,s$.
\item $I_{x_i} = \left\{ \begin{array}{rl}
	\mathbb{F}_2 &\mbox{ if }
	x_i \in R_i^* \\ 
	C_2(0) &\textnormal{ otherwise}
\end{array}\right.$ for $i=1, \ldots, s-3$ and 
$I_{x_i} = \left\{ \begin{array}{rl}
	\mathbb{F}_3 &\mbox{ if }
	x_i \in R_i^* \\ 
	C_3(0) &\textnormal{ otherwise}
\end{array}\right.$ for $i=s-2, s-1$ and $I_{x_s} = \mathbb{F}_4$.
\item $I_{x_i} = \left\{ \begin{array}{rl}
	\mathbb{F}_2 &\mbox{ if }
	x_i \in R_i^* \\ 
	C_2(0) &\textnormal{ otherwise}
\end{array}\right.$ for $i=1, \ldots, s-3$ and $I_{x_i}=\mathbb{F}_4$ for $i= s-2,s-1,s$.
\item $I_{x_i} = \left\{ \begin{array}{rl}
	\mathbb{F}_2 &\mbox{ if }
	x_i \in R_i^* \\ 
	C_2(0) &\textnormal{ otherwise}
\end{array}\right.$ for $i=1, \ldots , s-2$ and 
$I_{x_{s-1}} = \left\{ \begin{array}{rl}
	\mathbb{F}_3 &\mbox{ if }
	x_{s-1} \in R_{s-1}^* \\ 
	C_3(0) &\textnormal{ otherwise}
\end{array}\right.$ and $I_{x_s}= \mathcal{A}_9, \mathcal{B}_9, \mathcal{C}_9, \mathcal{D}_9, \mathcal{G}_9 , \mathcal{H}_9, \mathcal{I}_9, \mathcal{J}_9$.
\item $I_{x_1} = \mathcal{A}_4, \mathcal{B}_4, \mathcal{C}_4, \mathcal{D}_4, \mathcal{G}_4 , \mathcal{H}_4, \mathcal{I}_4, \mathcal{J}_4$, and $I_{x_i} = \left\{ \begin{array}{rl}
	\mathbb{F}_2 &\mbox{ if }
	x_i \in R_i^* \\ 
	C_2(0) &\textnormal{ otherwise}
\end{array}\right.$ for $i=2,\ldots, s-2$, and 
$I_{x_{s-1}} = \left\{ \begin{array}{rl}
	\mathbb{F}_{q_1} &\mbox{ if }
	x_{s-1} \in R_{s-1}^* \\ 
	C_{p_1}(0) &\textnormal{ otherwise}
\end{array}\right.$, 
$I_{x_s} = \left\{ \begin{array}{rl}
	\mathbb{F}_{q_2} &\mbox{ if }
	x_s \in R_s^* \\ 
	C_{p_2}(0) &\textnormal{ otherwise}
\end{array}\right.$,  where $p_1$ and $p_2$ are primes and $q_1$ and $q_2$ are prime powers such that 
\begin{align*}
3\leq|I_{x_{s-1}}|\leq |I_{x_s}| \leq |I_{x_{s-1}}|+ \sqrt{|I_{x_{s-1}}|\left(|I_{x_{s-1}}|-2\right)}.
\end{align*}
\item  $I_{x_i} = \left\{ \begin{array}{rl}
	\mathbb{F}_2 &\mbox{ if }
	x_i \in R_i^* \\ 
	C_2(0) &\textnormal{ otherwise}
\end{array}\right.$ for $i=1,\ldots,s-2$, and  
$I_{x_{s-1}} = \left\{ \begin{array}{rl}
	\mathbb{F}_{q_1} &\mbox{ if }
	x_{s-1} \in R_{s-1}^* \\ 
	C_{p_1}(0) &\textnormal{ otherwise}
\end{array}\right.$, 
$I_{x_s} = \left\{ \begin{array}{rl}
	\mathbb{F}_{q_2} &\mbox{ if }
	x_s \in R_s^* \\ 
	C_{p_2}(0) &\textnormal{ otherwise}
\end{array}\right.$,  where $p_1$ and $p_2$ are primes and $q_1$ and $q_2$ are prime powers such that
\begin{align*}
3\leq |I_{x_{s-1}}| \leq |I_{x_s}| \leq 2\left(|I_{x_{s-1}}|+\sqrt{|I_{x_{s-1}}| \left(|I_{x_{s-1}}|-2\right)}\right)-1.
\end{align*}
\item $\frac{I_{x_i}}{M_{x_i}}  = \left\{ \begin{array}{rl}
	\mathbb{F}_2 &\mbox{ if }
	x_i \in R_i^* \\ 
	C_2(0) &\textnormal{ otherwise}
\end{array}\right.$ for $i=1,\ldots s-1$ and $\frac{I_{x_s}}{M_{x_s}}= S$, where $S$ is a commutative ring such that $|S|=e \geq 3$ and
\begin{align*}
\prod_{i=1}^{s}|M_{x_i}| \leq 2 \left(e-1 +\sqrt{\left(e-2\right)e}\right).
\end{align*} 
\end{enumerate}
\end{theorem}

\begin{proof}
By Theorem~\ref{eigenvalue_of_cay(R,xR)}, ${\rm Cay}(R,xR^{*})$ is Ramanujan if and only if $|\lambda_{C}| \leq 2\sqrt{|xR^{*}|-1}$ for all $\lambda_C$ other than $\pm |xR^*|$, where $C \subseteq \{1, \ldots, s\}$ and 
\begin{align}
\lambda_{C}=(-1)^{|C|}\frac{|xR^{*}|}{\prod\limits_{i\in C}\left(\frac{|I_{x_i}|}{|M_{x_i}|}-1\right)}. \label{Eq1Chara_Ramanujan}
\end{align}

In general, we have $|I_{x_i}| \geq 2|M_{x_i}|$ for each $i=1, \ldots, s$. If $\frac{|I_{x_i}|}{|M_{x_i}|}=2$ for all $i=1,\cdots, s$ then $\pm |xR^{*}|$ are the only non-zero eigenvalue of ${\rm Cay}(R,xR^{*})$. Lemma~\ref{rings_of_order_P} implies that $\frac{I_{x_i}}{M_{x_i}}$ is isomorphic to either $\mathbb{F}_2$ (if $x_i\in R_i^*$) or $C_2(0)$ (if $x_i \not\in R_i^*$) for each $i =1,\ldots, s$. Hence if $\frac{|I_{x_i}|}{|M_{x_i}|}=2$ for all $i=1,\cdots, s$ then ${\rm Cay}(R,xR^{*})$ is Ramanujan if and only if the condition $(i)$ holds.

On the other hand, assume that  $\frac{|I_{x_i}|}{|M_{x_i}|}>2$ for some $i$. Let $t+1$ be the smallest integer such that $\frac{|I_{x_{t+1}}|}{|M_{x_{t+1}}|} > 2$ with $t \in \{0,1\ldots,s-1\}$. We have
\begin{equation}\label{ordering_on_the_order_of_IX}
2=\frac{|I_{x_i}|}{|M_{x_i}|}=\cdots =\frac{|I_{x_t}|}{|M_{x_t}|}<\frac{|I_{x_{t+1}}|}{|M_{x_{t+1}}|}\leq \cdots \leq \frac{|I_{x_s}|}{|M_{x_s}|}.
\end{equation}
If $C \subseteq \{1,\ldots, t\}$ then Equation~(\ref{Eq1Chara_Ramanujan}) and Equation~(\ref{ordering_on_the_order_of_IX}) implies $|\lambda_C|= |xR^{*}|$. Similarly, if $C\cap \{t+1,\ldots,s \} \neq \emptyset$ then Equation~(\ref{Eq1Chara_Ramanujan}) and Equation~(\ref{ordering_on_the_order_of_IX}) implies $|\lambda_{C}|=|\lambda_{C\cap \{t+1,\ldots, s\}}| \leq |\lambda_{\{t+1\}}|<|xR^{*}|$. Thus $|\lambda_{\{t+1\}}|$ is the second largest absolute value of eigenvalue of ${\rm Cay}(R,xR^{*})$. Hence,  ${\rm Cay}(R,xR^{*})$ is Ramanujan if and only if $|\lambda_{\{t+1\}}| \leq 2 \sqrt{|xR^{*}|-1}$, equivalent to say, 
 \begin{equation}\label{Main_Rama_con}
  \frac{|xR^{*}|}{\left(\frac{|I_{x_{t+1}}|}{|M_{x_{t+1}}|}-1\right)} \leq 2\sqrt{|xR^{*}|-1}.
 \end{equation}
 Using $\sqrt{|xR^{*}|-1} <  \sqrt{|xR^{*}|}$, Equation~(\ref{Main_Rama_con}) implies
 \begin{equation} \label{Eq2Chara_Ramanujan}
 |xR^{*}| < 4 \left(\frac{|I_{x_{t+1}}|}{|M_{x_{t+1}}|}-1\right)^2.
 \end{equation}
 Using Equation (\ref{card_of_xR*}), Equation~(\ref{Eq2Chara_Ramanujan}) implies
 \begin{equation}\label{if_condition_for_ramanujan}
 \prod_{i=1}^s|M_{x_i}|\prod_{i=t+2}^s \left( \frac{|I_{x_i}|}{|M_{x_i}|}-1\right) < 4\left(\frac{|I_{x_{t+1}}|}{|M_{x_{t+1}}|}-1\right).
 \end{equation}
 
If $s \geq t+4$ then using Equation~(\ref{ordering_on_the_order_of_IX}) we get  $\prod\limits_{i=t+2}^s \left( \frac{|I_{x_i}|}{|M_{x_i}|}-1\right) \geq 4\left(\frac{|I_{x_{t+1}}|}{|M_{x_{t+1}}|}-1\right)$. Thus Equation (\ref{if_condition_for_ramanujan}) does not hold. Therefore if $s\geq t+4$, then there is no commutative ring $R$ exist such that Cayley graph ${\rm Cay}(R,xR^{*})$ is Ramanujan. Now we have remaining three cases, $s=t+3$, $s=t+2$, and $s=t+1$.\\
 \textbf{Case 1}: $s=t+3$\\
 We rewrite the Equation (\ref{if_condition_for_ramanujan})
\begin{equation} \label{S=t+3_case}
  \prod_{i=1}^s |M_{x_i}|\left( \frac{|I_{x_{t+2}}|}{|M_{x_{t+2}}|}-1\right)\left(\frac{|I_{x_{t+3}}|}{|M_{x_{t+3}}|}-1\right)<4\left(\frac{|I_{x_{t+1}}|}{|M_{x_{t+1}}|}-1\right).
\end{equation}
Note that $ \left( \frac{|I_{x_{t+2}}|}{|M_{x_{t+2}}|}-1\right)\left(\frac{|I_{x_{t+3}}|}{|M_{x_{t+3}}|}-1\right) \geq 2\left(\frac{|I_{x_{t+1}}|}{|M_{x_{t+1}}|}-1\right)$. If $\prod_{i=1}^{s}|M_{x_i}| \geq 2$ then equation (\ref{S=t+3_case}) does not hold. Therefore, if  $\prod_{i=1}^{s}|M_{x_i}| \geq 2$ then there is no commutative ring $R$ exist such that ${\rm Cay}(R,xR^{*})$ is Ramanujan.
Assume that $\prod_{i=1}^{s}|M_{x_i}| =1$. Equation (\ref{S=t+3_case}) implies
\begin{equation}\label{Simplified_s=t+3_case}
\left(|I_{x_{t+2}}|-1\right)\left(|I_{x_{t+3}}|-1\right)<4\left(|I_{x_{t+1}}|-1\right). 
\end{equation} We have the following two cases:\\ 
Case 1.1: $|I_{x_{t+1}}|=|I_{x_{t+2}}|$\\
Equation (\ref{Simplified_s=t+3_case}) implies $|I_{x_{t+3}}| \leq 4$. From equation (\ref{ordering_on_the_order_of_IX}), we have the following three cases:
\begin{align*}
(a)&~~~ |I_{x_{t+1}}| = |I_{x_{t+2}}| = |I_{x_{t+3}}|=3\\
(b)&~~~|I_{x_{t+1}}| = |I_{x_{t+2}}| =3, |I_{x_{t+3}}|=4\\
(c)&~~~|I_{x_{t+1}}| = |I_{x_{t+2}}| = |I_{x_{t+3}}|=4
\end{align*} 
The cases $(a),(b),(c)$ satisfy to Equation (\ref{Main_Rama_con}). Hence, in this particular case, the Cayley graph ${\rm Cay}(R,xR^{*})$ is Ramanujan if and only if  $R$ satisfies any one of the cases (a), (b), (c). Note that  $\prod\limits_{i=1}^{s}|M_{x_i}| =1$, and so $|M_{x_i}|=1$ for all $i=1,\ldots,s$. By Equation \ref{ordering_on_the_order_of_IX}, we have $|I_{x_i}|=2$ for all $i=1,\ldots,t$. Therefore, $I_{x_i} = \mathbb{F}_2$ (if $x_i \in R_i^*$) or $C_2(0)$ (if $x_i \not\in R_i^*$) for all $i=1,\ldots,t$. Using Part $(iv)$ and Part $(v)$ of Lemma~\ref{PropertiesOfMx} and Lemma~\ref{rings_of_order_P}, we write the cases in the equivalent way:
\begin{align*}
(a)&~~~ \textnormal{$I_{x_i} = \mathbb{F}_3$ (if $x_i \in R_i^*$) or $C_3(0)$ (if $x_i \not\in R_i^*$) for all $i=t+1,t+2,t+3$}. \\
(b)&~~~  \textnormal{$I_{x_i} = \mathbb{F}_3$ (if $x_i \in R_i^*$) or $C_3(0)$ (if $x_i \not\in R_i^*$) for all $i=t+1,t+2$, and $I_{x_{t+3}}=\mathbb{F}_4$}. \\
(c)&~~~ \textnormal{$I_{x_i} = \mathbb{F}_4$ for all $i=t+1,t+2,t+3$}.
\end{align*} 
Hence, in this particular case, the Cayley graph ${\rm Cay}(R,xR^{*})$ is Ramanujan if and only if  $R$ satisfies any one of the condition $(ii), (iii)$ and $(iv)$.\\
Case 1.2: $|I_{x_{t+1}}| < |I_{x_{t+2}}|$\\
Equation (\ref{Simplified_s=t+3_case}) is equivalent to $|I_{x_{t+1}}|=3, |I_{x_{t+2}}|=4$ and $|I_{x_{t+3}}|=4$. But it does not satisfy by Equation (\ref{Main_Rama_con}). Hence, in this particular case,  there is no commutative ring $R$ exist such that ${\rm Cay}(R,xR^{*})$ is Ramanujan.\\
\textbf{Case 2}: $s=t+2$\\
In this case, Equation (\ref{if_condition_for_ramanujan}) reduce to
\begin{align} \label{1st_cond_S=t+2}
\prod_{i=1}^s|M_{x_i}|\left(\frac{|I_{x_{t+2}}|}{|M_{x_{t+2}}|}-1 \right) < 4\left(\frac{|I_{x_{t+1}}|}{|M_{x_{t+1}}|} -1\right).
\end{align}
Observe that if $\prod_{i=1}^s|M_{x_i}| \geq 4$ then Equation (\ref{1st_cond_S=t+2}) does not hold. Hence,  if $\prod_{i=1}^s|M_{x_i}| \geq 4$ then there is no commutative ring $R$ exist such that ${\rm Cay}(R,xR^{*})$ is Ramanujan. Therefore, we have either  $\prod\limits_{i=1}^s|M_{x_i}|=3$,  $\prod\limits_{i=1}^s|M_{x_i}|=2$, or  $\prod\limits_{i=1}^s|M_{x_i}|=1$.\\
Case 2.1: $\prod\limits_{i=1}^s|M_{x_i}|=3$\\
In this case Equation (\ref{Main_Rama_con}) reduce to
\begin{align} \label{S=t+2_1st_case }
3\left(\frac{|I_{x_{t+2}}|}{|M_{x_{t+2}}|}-1\right) \leq 2\sqrt{3\left(\frac{|I_{x_{t+1}}|}{|M_{x_{t+1}}|}-1\right)\left(\frac{|I_{x_{t+2}}|}{|M_{x_{t+2}}|}-1\right)-1}.
\end{align}
Since $\prod\limits_{i=1}^s|M_{x_i}|=3$, there exist  unique $j \in \{1,\ldots,s\}$ such that $|M_{x_j}|=3$. By Part $(vii)$ of Lemma~\ref{PropertiesOfMx}, we get $|I_{x_{j}}| \leq 9$. Using the fact that $M_{x_j}$ is a proper ideal of $I_{x_{j}}$, we get $|I_{x_j}|=9$. And so $\frac{|I_{x_{j}}|}{|M_{x_j}|}=3$ implies $j = t+1$ or $j=t+2$. Therefore we have two cases:
\begin{align*}
(a) &~~~ \textnormal{$|M_{x_{i}}|=1$ for all $ i \in \{1,\ldots, t\}$, $|M_{x_{t+1}}|=3$, and $|M_{x_{t+2}}|=1$.}\\
(b) &~~~ \textnormal{$|M_{x_{i}}|=1 $ for all $ i \in \{1,\ldots, t\}$, $|M_{x_{t+1}}|=1$, and $|M_{x_{t+2}}|=3$.} 
\end{align*} 

In the case $(a)$, $\mathcal{A}_9,\mathcal{B}_9,\mathcal{C}_9,\mathcal{D}_9,\mathcal{G}_9,\mathcal{H}_{9},\mathcal{I}_9,\mathcal{J}_9$ are the only possibilities for $I_{x_{t+1}}$ (follows from theorem \ref{rings_of_order_p2}) and $I_{x_{t+2}} = \mathbb{F}_q $ (if $x_{t+2} \in R_{t+2}^*$) or $C_p(0)$ (if $x_{t+2} \not\in R_{t+2}^*$), where $p$ is a prime and $q$ is some power of prime (see Part $(v)$ of Lemma~\ref{PropertiesOfMx}). Rewrite the Equation (\ref{S=t+2_1st_case }) in another equivalent form
\begin{align}
\frac{|I_{x_{t+2}}|}{|M_{x_{t+2}}|} \leq \frac{2}{3}\left(\frac{|I_{x_{t+1}}|}{|M_{x_{t+1}}|}+\sqrt{\left(\frac{|I_{x_{t+1}}|}{|M_{x_{t+1}}|}-1\right)\left(\frac{|I_{x_{t+1}}|}{|M_{x_{t+1}}|}\right)}\right)+\frac{1}{3}. \label{Eq13_Ramanujan_Character}
\end{align}

In this case, we have $|M_{x_{t+2}}|=1,|M_{x_{t+1}}|=3$, and $|I_{x_{t+1}}|=9$. Equation~\ref{Eq13_Ramanujan_Character} implies that  the Cayley graph ${\rm Cay}(R,xR^*)$ is Ramanujan if and only if $|I_{x_{t+2}}|=3$, equivalent to say, $R$ satisfies the condition $(v)$. In the case (b), we have$\frac{|I_{x_{t+2}}|}{|M_{x_{t+2}}|}=3$. Therefore, Equation (\ref{ordering_on_the_order_of_IX}) implies  $|I_{x_{t+1}}|=3$. Again, the Cayley graph ${\rm Cay}(R,xR^*)$ is Ramanujan if and only if $R$ satisfies the condition $(v)$.\\
Case 2.2 $\prod\limits_{i=1}^s |M_{x_{i}}|=2$\\
In this case the Equation (\ref{Main_Rama_con}) reduce to
\begin{align}\label{s=t+2_case2}
2\left(\frac{|I_{x_s}|}{|M_{x_s}|}-1\right)\leq 2\sqrt{2\left(\frac{|I_{x_s}|}{|M_{x_s}|}-1\right)\left(\frac{|I_{x_{s-1}}|}{|M_{x_{s-1}}|}-1\right)-1}.
\end{align}
As $\prod\limits_{i=1}^s|M_{x_i}|=2$, therefore there exist unique $j \in \{1, \ldots, s\}$ such that $|M_{x_j}|=2$. By Part $(vii)$ of Lemma~\ref{PropertiesOfMx} we get $|I_{x_j}| \leq 4$. Using the fact that $M_{x_j}$ is a proper ideal of $I_{x_j}$, therefore we get $|I_{x_j}|=4$. And so $\frac{|I_{x_j}|}{|M_{x_j}|} =2$ implies $ j \in \{1, \ldots, t\}$. Without loss of generality assume that $j=1$, and so $|M_{x_1}|=2$ and $|M_{x_i}|=1$ for all $i \in \{2 \ldots s\}$. In Theorem \ref{rings_of_order_p2}, $\mathcal{A}_4,\mathcal{B}_4,\mathcal{C}_4,\mathcal{D}_4,\mathcal{G}_4,\mathcal{H}_{4},\mathcal{I}_4,\mathcal{J}_4$ are the only commutative rings of order $4$ such that the cardinality of maximal ideal is $2$. Hence we have $I_{x_1} =\mathcal{A}_4,\mathcal{B}_4,\mathcal{C}_4,\mathcal{D}_4,\mathcal{G}_4,\mathcal{H}_{4},\mathcal{I}_4,\mathcal{J}_4$, $I_{x_{i}}= \mathbb{F}_2$ (if $x_i\in R_i^*$) or $C_2(0)$ (if $x_i\not\in R_i^*$) for each $ i \in \{2 \ldots t\}$, and $I_{x_{t+1}}=\mathbb{F}_{q_1}$ (if $x_{t+1}\in R_{t+1}^*$) or $C_{p_1}(0)$ (if $x_{t+1}\not \in R_{t+1}^*$), and $I_{x_{t+2}}=\mathbb{F}_{q_2}$  (if $x_{t+2}\in R_{t+2}^*$) or $ C_{p_2}(0)$ (if $x_{t+2}\not\in R_{t+2}^*$), where $p_1$ and $p_2
$ are primes and $q_1$ and $q_2$ are prime powers. By equation (\ref{s=t+2_case2}), ${\rm  Cay}(R,xR^{*})$ is Ramanujan if and only if it satisfies the following equation\\
\begin{align*}
2\left(|I_{x_{t+2}}|-1\right) \leq 2 \sqrt{2\left(|I_{x_{t+1}}|-1\right)\left(|I_{x_{t+2}}|-1\right)-1},
\end{align*}
which is equivalent to 
\begin{align*}
|I_{x_{t+2}}| \leq |I_{x_{t+1}}|+ \sqrt{|I_{x_{t+1}}|\left(|I_{x_{t+1}}|-2\right)}.
\end{align*} 
Hence this gives condition $(vi)$.\\
Case 2.3 $\prod\limits_{i=1}^s|M_{x_i}|=1$\\
In this case, $I_{x_i} = \mathbb{F}_2$ (if $x_i\in R_i^*$) or $ C_2(0)$ (if $x_i\not\in R_i^*$)  for each $i \in \{1 \ldots t\}$, $I_{x_{t+1}}=\mathbb{F}_{q_1}$ (if $x_{t+1}\in R_{t+1}^*$) or $C_{p_1}(0)$ (if $x_{t+1}\not \in R_{t+1}^*$), and $I_{x_{t+2}}=\mathbb{F}_{q_2}$  (if $x_{t+2}\in R_{t+2}^*$) or $ C_{p_2}(0)$ (if $x_{t+2}\not\in R_{t+2}^*$), where $p_1$ and $p_2$ are primes and $q_1$ and $q_2$ are prime powers. In this case ${\rm Cay}(R,xR^{*})$ is Ramanujan if and only if it satisfies equation (\ref{Main_Rama_con}), that is,
\begin{align*}
|I_{x_{t+2}}|-1 \leq 2 \sqrt{\left(|I_{x_{t+1}}|-1\right)\left(|I_{x_{t+2}}|-1\right)-1},
\end{align*} 
which is equivalent to 
\begin{align*}
|I_{x_{t+2}}| \leq 2\left(|I_{x_{t+1}}|+\sqrt{|I_{x_{t+1}}| \left(|I_{x_{t+1}}|-2\right)}\right)-1.
\end{align*}
Hence this gives condition $(vii)$.\\
\textbf{Case 3}: $s=t+1$\\
In this case Equation (\ref{Main_Rama_con}) reduce to
\begin{align*}
\prod_{i=1}^s|M_{x_i}| \leq 2 \sqrt{\prod_{i=1}^{t+1}|M_{x_{i}}|\left(\frac{|I_{x_s}|}{|M_{x_s}|}-1\right)-1},
\end{align*}
which is equivalent to
\begin{align*}
\prod_{i=1}^s|M_{x_i}| \leq 2\left(\sqrt{\frac{|I_{x_s}|}{|M_{x_s}|}\left(\frac{|I_{x_s}|}{|M_{x_s}|}-2\right)}+\frac{|I_{x_s}|}{|M_{x_s}|}-1\right).
\end{align*}
Hence ${\rm Cay}(R,xR^{*})$ will be Ramanujan as given in $(viii)$.\\
\end{proof}

\begin{corollary}
Let $R$ be a local ring and $x$ be a non-zero element of $R$. Then ${\rm Cay}(R,xR^{*})$ is Ramanujan if and only if one of the following holds:
\begin{enumerate}[label =(\roman*)]
\item $\frac{I_{x}}{M_{x}}  = \left\{ \begin{array}{rl}
	\mathbb{F}_2 &\mbox{ if }
	x \in R^* \\ 
	C_2(0) &\textnormal{ otherwise}
\end{array}\right.$, equivalent to say, $|I_x|=2|M_x|$.
\item $|I_x| \geq \left(\frac{|M_x|}{2}+1\right)^2$.
\end{enumerate}
\end{corollary}
\begin{proof} Take $s=1$ in Theorem \ref{Ramanujan_characterization}. Therefore, ${\rm Cay}(R,xR^{*})$ is Ramanujan if and only if either condition $(i)$ or condition $(viii)$ of Theorem \ref{Ramanujan_characterization} holds. The proof follows from the fact that condition $(viii)$ of Theorem \ref{Ramanujan_characterization} is equivalent to $|I_x| \geq \left(\frac{|M_x|}{2}+1\right)^2$.
\end{proof}

Let $R$ be a finite commutative ring, $x$ be a non-zero element of $R$, and $P=\{ i \colon  i \in \{ 1, \ldots, s \} \text{ and } x_i \neq \bold{0}\}$.  Without loss of generality, assume that $P=\{ h_1, \ldots , h_r\}$ with $h_1\leq \ldots \leq h_r$. Define $R':=R_{h_1}\times \cdots \times R_{h_r}$ and $y:=(x_{h_1} , \ldots , x_{h_r})$.  Since if $i \not\in P$ then the $(0,1)$-adjacency matrix of ${\rm Cay}(R_{i},x_{i}R_{i}^{*})$ is the identity matrix, by Theorem~\ref{Eigenvalues_of_tensor_product_of_graphs} both graphs $\otimes_{i=1}^{s} {\rm Cay}(R_{i},x_{i}R_{i}^{*})$ and $\otimes_{i\in P} {\rm Cay}(R_{i},x_{i}R_{i}^{*})$ have same set of eigenvalues, but their multiplicities may be different. By Theorem~\ref{R_as_a_product of local rings},  both graphs ${\rm Cay}(R,xR^{*})$ and $\otimes_{i=1}^{s} {\rm Cay}(R_{i},x_{i}R_{i}^{*})$ are isomorphic. Again, using Theorem~\ref{R_as_a_product of local rings}, both graphs ${\rm Cay}(R',yR'^{*})$ and $\otimes_{i\in P} {\rm Cay}(R_{i},x_{i}R_{i}^{*})$ are isomorphic. Hence both graphs ${\rm Cay}(R,xR^{*})$ and ${\rm Cay}(R',yR'^{*})$ have same set of eigenvalues, but their multiplicities may be different. Using the fact that ${\rm Cay}(R,xR^{*})$ is Ramanujan if and only if ${\rm Cay}(R',yR'^{*})$ is Ramanujan, we conclude the following result.

\begin{theorem} \label{arb_Ramanujan_characterization}
Let $R$ be a finite commutative ring and $x$ be a non-zero element of $R$. Then ${\rm Cay}(R,xR^{*})$ is Ramanujan if and only if one of the following holds;
\begin{enumerate}[label =(\roman*)]
\item $\frac{I_{y_i}}{M_{y_i}}  = \left\{ \begin{array}{rl}
	\mathbb{F}_2 &\mbox{ if }
	y_i \in R_i^* \\ 
	C_2(0) &\textnormal{ otherwise}
\end{array}\right.$ for $i=1,\ldots, r$.
\item $I_{y_i} = \left\{ \begin{array}{rl}
	\mathbb{F}_2 &\mbox{ if }
	y_i \in R_i^* \\ 
	C_2(0) &\textnormal{ otherwise}
\end{array}\right.$ for $i = 1, \ldots, r-3$ and 
$I_{y_i} = \left\{ \begin{array}{rl}
	\mathbb{F}_3 &\mbox{ if }
	y_i \in R_i^* \\ 
	C_3(0) &\textnormal{ otherwise}
\end{array}\right.$ for $i=r-2,r-1,r$.
\item $I_{y_i} = \left\{ \begin{array}{rl}
	\mathbb{F}_2 &\mbox{ if }
	y_i \in R_i^* \\ 
	C_2(0) &\textnormal{ otherwise}
\end{array}\right.$ for $i=1, \ldots, r-3$ and 
$I_{y_i} = \left\{ \begin{array}{rl}
	\mathbb{F}_3 &\mbox{ if }
	y_i \in R_i^* \\ 
	C_3(0) &\textnormal{ otherwise}
\end{array}\right.$ for $i=r-2, r-1$ and $I_{y_r} = \mathbb{F}_4$.
\item $I_{y_i} = \left\{ \begin{array}{rl}
	\mathbb{F}_2 &\mbox{ if }
	y_i \in R_i^* \\ 
	C_2(0) &\textnormal{ otherwise}
\end{array}\right.$ for $i=1, \ldots, r-3$ and $I_{y_i}=\mathbb{F}_4$ for $i= r-2,r-1,r$.
\item $I_{y_i} = \left\{ \begin{array}{rl}
	\mathbb{F}_2 &\mbox{ if }
	y_i \in R_i^* \\ 
	C_2(0) &\textnormal{ otherwise}
\end{array}\right.$ for $i=1, \ldots , r-2$ and 
$I_{y_{r-1}} = \left\{ \begin{array}{rl}
	\mathbb{F}_3 &\mbox{ if }
	y_{r-1} \in R_{r-1}^* \\ 
	C_3(0) &\textnormal{ otherwise}
\end{array}\right.$ and $I_{y_r}= \mathcal{A}_9, \mathcal{B}_9, \mathcal{C}_9, \mathcal{D}_9, \mathcal{G}_9 , \mathcal{H}_9, \mathcal{I}_9, \mathcal{J}_9$.
\item $I_{y_1} = \mathcal{A}_4, \mathcal{B}_4, \mathcal{C}_4, \mathcal{D}_4, \mathcal{G}_4 , \mathcal{H}_4, \mathcal{I}_4, \mathcal{J}_4$, and $I_{y_i} = \left\{ \begin{array}{rl}
	\mathbb{F}_2 &\mbox{ if }
	y_i \in R_i^* \\ 
	C_2(0) &\textnormal{ otherwise}
\end{array}\right.$ for $i=2,\ldots, r-2$, and 
$I_{y_{r-1}} = \left\{ \begin{array}{rl}
	\mathbb{F}_{q_1} &\mbox{ if }
	y_{r-1} \in R_{r-1}^* \\ 
	C_{p_1}(0) &\textnormal{ otherwise}
\end{array}\right.$, 
$I_{y_r} = \left\{ \begin{array}{rl}
	\mathbb{F}_{q_2} &\mbox{ if }
	y_r \in R_r^* \\ 
	C_{p_2}(0) &\textnormal{ otherwise}
\end{array}\right.$,  where $p_1$ and $p_2$ are primes and $q_1$ and $q_2$ are prime powers such that 
\begin{align*}
3\leq|I_{y_{r-1}}|\leq |I_{y_r}| \leq |I_{y_{r-1}}|+ \sqrt{|I_{y_{r-1}}|\left(|I_{y_{r-1}}|-2\right)}.
\end{align*}
\item  $I_{y_i} = \left\{ \begin{array}{rl}
	\mathbb{F}_2 &\mbox{ if }
	y_i \in R_i^* \\ 
	C_2(0) &\textnormal{ otherwise}
\end{array}\right.$ for $i=1,\ldots,r-2$, and  
$I_{y_{r-1}} = \left\{ \begin{array}{rl}
	\mathbb{F}_{q_1} &\mbox{ if }
	y_{r-1} \in R_{r-1}^* \\ 
	C_{p_1}(0) &\textnormal{ otherwise}
\end{array}\right.$, 
$I_{y_r} = \left\{ \begin{array}{rl}
	\mathbb{F}_{q_2} &\mbox{ if }
	y_r \in R_r^* \\ 
	C_{p_2}(0) &\textnormal{ otherwise}
\end{array}\right.$,  where $p_1$ and $p_2$ are primes and $q_1$ and $q_2$ are prime powers such that
\begin{align*}
3\leq |I_{y_{r-1}}| \leq |I_{y_r}| \leq 2\left(|I_{y_{r-1}}|+\sqrt{|I_{y_{r-1}}| \left(|I_{y_{r-1}}|-2\right)}\right)-1.
\end{align*}
\item $\frac{I_{y_i}}{M_{y_i}}  = \left\{ \begin{array}{rl}
	\mathbb{F}_2 &\mbox{ if }
	y_i \in R_i^* \\ 
	C_2(0) &\textnormal{ otherwise}
\end{array}\right.$ for $i=1,\ldots r-1$ and $\frac{I_{y_r}}{M_{y_r}}= S$, where $S$ is a commutative ring such that $|S|=e \geq 3$ and
\begin{align*}
\prod_{i=1}^{r}|M_{y_i}| \leq 2 \left(e-1 +\sqrt{\left(e-2\right)e}\right).
\end{align*} 
\end{enumerate}
where $y_i$ for $i=1, \ldots r$ is defined as above.
\end{theorem}
\begin{proof} Proof follows from previous theorem.
\end{proof}

\section*{Acknowledgements} The first author is supported by Junior Research Fellowship from CSIR, Government of India (File No. 09/1020(15619)/2022-EMR-I).


\begin{thebibliography}{10}
\bibitem{akhtar2009unitary}
R. Akhtar, M. Boggess, T. J. Henderson, I. Jim{\'e}nez, R. Karpman, A. Kinzel, and D. Pritikin.
\newblock On the unitary Cayley graph of a finite ring.
\newblock {\em The Electronic Journal of Combinatorics}, 16:R117, 2009.

\bibitem{atiyah2018introduction}
M. Atiyah.
\newblock {\em Introduction to commutative algebra}.
\newblock CRC Press, 2018.


\bibitem{davidoff2003elementary}
G. P. Davidoff, P. Sarnak, and A. Valette.
\newblock {\em Elementary number theory, group theory, and Ramanujan graphs},
\newblock Cambridge university press Cambridge, Volume~55, 2003.

\bibitem{droll2010classification}
A. Droll.
\newblock A classification of ramanujan unitary cayley graphs.
\newblock {\em The Electronic Journal of Combinatorics}, 17:N29, 2010.

\bibitem{fine1993classification}
B. Fine.
\newblock Classification of finite rings of order p2.
\newblock {\em Mathematics Magazine}, 248--252, 1993.

\bibitem{ganesan1964properties}
N.~Ganesan.
\newblock Properties of rings with a finite number of zero divisors.
\newblock {\em Mathematische Annalen}, 157(3):215--218, 1964.

\bibitem{godsil2001algebraic}
C. Godsil and G. Royle.
\newblock {\em Algebraic graph theory}, volume 207.
\newblock Springer Science \& Business Media, 2001.

\bibitem{mmmmmNew}
I. Gutman.
\newblock The energy of a graph.
\newblock In {\em Ber. Math.-Statist. Sekt. Forsch. Graz, (100- 105):Ber. No.
  103, 22, 1978. 10. Steierm{\"a}rkisches Mathematisches Symposium (Stift Rein,
  Graz, 1978)}.

\bibitem{gutman2001energy}
I. Gutman.
\newblock The energy of a graph: old and new results.
\newblock In {\em Algebraic Combinatorics and Applications: Proceedings of the
  Euroconference, Algebraic Combinatorics and Applications (ALCOMA), held in
  G{\"o}{\ss}weinstein, Germany, September 12--19, 1999}, 196--211.
  Springer.

\bibitem{ilic2009energy}
A. Ili{\'c}.
\newblock The energy of unitary cayley graphs.
\newblock {\em Linear Algebra and its Applications}, 431(10):1881--1889, 2009.

\bibitem{kiani2011energy}
D. Kiani, M. M. H. Aghaei, Y. Meemark, and B. Suntornpoch.
\newblock Energy of unitary cayley graphs and gcd-graphs.
\newblock {\em Linear Algebra and its Applications}, 435(6):1336--1343, 2011.

\bibitem{liu2012spectral}
X. Liu and S. Zhou.
\newblock Spectral properties of unitary cayley graphs of finite commutative
  rings.
\newblock {\em The Electronic Journal of Combinatorics}, 19(4):P13, 2012.

\bibitem{liu2022eigenvalues}
X. Liu and S. Zhou.
\newblock Eigenvalues of cayley graphs.
\newblock {\em The Electronic Journal of Combinatorics}, 29(2):P2.9, 2022.

\bibitem{murty2020ramanujan}
M.~R. Murty.
\newblock Ramanujan graphs: An introduction.
\newblock {\em Indian Journal Discret. Math}, 6:91--127, 2020.

\bibitem{west2001introduction}
D. B. West.
\newblock {\em Introduction to graph theory}, Volume~2.
\newblock  Upper Saddle River, Prentice hal 2001.


\end{thebibliography}
\end{document}